\documentclass[10pt,a4paper]{amsart}
\usepackage{amscd,amssymb,amsmath,latexsym}

\usepackage{amsfonts}
\usepackage{graphicx}%
\usepackage[active]{srcltx}
\newtheorem{definition}{Definition}[section]
\newtheorem{notation}[definition]{Notation}
\newtheorem{rema}[definition]{Remark}
\newtheorem{exa}[definition]{Example}
\newtheorem{obs}[definition]{Observation}

\newtheorem{lemma}[definition]{Lemma}
\newtheorem{proposition}[definition]{Proposition}
\newtheorem{theorem}[definition]{Theorem}
\newtheorem{corollary}[definition]{Corollary}

\newenvironment{remark}%
{\begin{rema}\rm}%
{\end{rema} }

{\begin{exa}\rm}%
{\end{exa} }

\newenvironment{observation}%
{\begin{obs}\rm}%
{\end{obs} }

\newenvironment{ac}{\noindent{\bf Acknowledgements:}}{}

\newcommand{\Sres}{{\operatorname{Sres}}}
\newcommand{\Res}{{\operatorname{Res}}}
\newcommand{\Syl}{{\operatorname{Syl}}}
\newcommand{\mSyl}{{\operatorname{MSyl}}}

\newcommand{\coeff}{{\operatorname{coeff}}}

\def\sg{\mbox{sg}}
\newcommand{\bfa}{{\boldsymbol{a}}}

\newcommand{\bfx}{X}
\newcommand{\bfy}{Y}

\def\RR{{\mathcal R}}

\def\C{{\mathbb C}}

\def\Q{{\mathbb Q}}

\def\Z{{\mathbb Z}}

\def\sg{\mbox{sg}}
\def\Res{{\rm{Res}}}

\begin{document}

\title[Symmetric interpolation  and  Sylvester sums]{Symmetric interpolation, Exchange Lemma \\and  Sylvester sums}

\author{Teresa Krick}
\address{Departamento de Matem\'atica, Facultad de
Ciencias Exactas y Naturales  and IMAS,
CONICET, Universidad de Buenos Aires,  Argentina} \email{krick@dm.uba.ar}
\urladdr{http://mate.dm.uba.ar/\~\,krick}

\author{Agnes Szanto}
\address{Department of Mathematics, North Carolina State
University, Raleigh, NC 27695 USA}
\email{aszanto@ncsu.edu}
\urladdr{www4.ncsu.edu/\~\,aszanto }

\author{Marcelo Valdettaro}
\address{Departamento de Matem\'atica, Facultad de
Ciencias Exactas y Naturales, Universidad de Buenos Aires, Argentina}
\email{mvaldettaro@gmail.com}

\begin{abstract}
The theory of symmetric multivariate Lagrange interpolation is a beautiful but rather unknown tool that has many applications. Here we derive from it an Exchange Lemma that allows to explain in a simple and natural way the full  description of  the  double sum expressions introduced by Sylvester in 1853 in terms of subresultants and their B\'ezout coefficients.
\end{abstract}

\date{\today}
\thanks{
Teresa Krick is partially
suported by ANPCyT PICT-2013-0294  and UBACyT
2014-2017-20020130100143BA, and Agnes Szanto was partially supported by NSF grant
CCF-1217557.}\maketitle

\maketitle

\section{Introduction}

 W. Chen and J. Louck proved in \cite[Th.2.1]{CL96} a beautiful interpolation result which describes the Lagrange interpolation basis for all multivariate symmetric polynomials in $m-d$ variables of multidegree bounded by $(d,\dots,d)$, for $0\le d<m$, see Section~\ref{section1} below for a precise statement. They use their result to recover identities involving symmetric functions, generalizing for instance the following  polynomial identity for a finite set $A=\{\alpha_1,\dots,\alpha_m\}$ contained in a field, and a finite
set of variables $X=\{x_1\dots, x_{m-d}\}$:

$$x_1\cdots x_{m-d} = \sum_{{A'\subset A, |A'|=d}} \Big(\prod_{\alpha_j \notin A'} \alpha_j\Big)\  \dfrac{ \prod_{x_j\in X, \alpha_i\in A'}(x_j-\alpha_i)}{ \prod_{\alpha_j\notin A', \alpha_i\in A'}(\alpha_j-\alpha_i)}.$$

Here we take another direction and derive from this symmetric interpolation the following Exchange Lemma (see Corollary~\ref{exchange2} below): given $0\le d\le m$ and $0\le r\le m-d$, for any finite sets $A$ and $B$ contained in a field,  satisfying $|A|=m$ and $|B|\ge d$, and any set of variables $X$ with $|X|=r$, one has the following polynomial identity
\[\sum_{\substack{A^{\prime}\subset A\\|A^{\prime}|=d}}\mathcal{R}(A\backslash A^{\prime},B)\frac{\mathcal{R}(\bfx,A^{\prime})}{\mathcal{R}(A^{\prime},A\backslash A^{\prime})}= (-1)^{d(m-d)} \sum_{\substack{B^{\prime}\subset B\\|B^{\prime}|=d}}\mathcal{R}(A,B\backslash B^{\prime})\frac{\mathcal{R}(\bfx,B^{\prime})}{\mathcal{R}(B^{\prime},B\backslash B^{\prime})} ,\]
where $ \mathcal{R}(Y,Z):=\prod_{y\in Y,z\in
Z}(y-z)
$ (and $ \mathcal{R}(Y,Z):=1
$ if $Y$ or $Z$ is empty).

In the particular case that  $|B|=d$, our Exchange Lemma   reads \[\sum_{{A^{\prime}\subset A,|A^{\prime}|=d}}\mathcal{R}(A\backslash A^{\prime},B)\frac{\mathcal{R}(\bfx,A^{\prime})}{\mathcal{R}(A\backslash A^{\prime},A^{\prime})}= \mathcal{R}(\bfx,B), \]
which is exactly the statement of \cite[Lem.Rt1]{Las}, proved there using Schur functions. However, it does not seem possible to directly recover our Exchange Lemma from Lascoux result.

In these pages, we use symmetric interpolation and the Exchange Lemma to show in a very natural way the different and somehow puzzling relationships between Sylvester single and double sums, subresultants and their B\'ezout coefficients.

{\em Double sums} were introduced in  \cite{sylv}: for $A=\{\alpha_1,\dots,\alpha_m\}$ and $B=\{\beta_1,\dots,\beta_n\}$ contained in a field, he defined for $0\le p\le m$, $0\le q\le n$

\[
 \Syl_{p,q}(A,B)(x):=\sum_{\substack{A^{\prime
}\subset A,\,B^{\prime}\subset
B\\|A^{\prime}|=p,\,|B^{\prime}|=q}}\mathcal{R}(A^{\prime},B^{\prime})\,
\mathcal{R}(A\backslash A^{\prime},B\backslash
B^{\prime})\,\frac{\mathcal{R}(x,A^{\prime
})\,\mathcal{R}(x,B^{\prime})}{\mathcal{R}(A^{\prime},A\backslash A^{\prime
})\,\mathcal{R}(B^{\prime},B\backslash B^{\prime})},\]
which  is a polynomial in $x$ of degree bounded by $p+q=:d$. When $p=0$ or $q=0$, the resulting expression is called a {\em single sum}.

Note that the previous Exchange Lemma applied to $X=\{x\}$ reads
\begin{equation}\label{relsingle} \Syl_{d,0}(A,B)(x)=(-1)^{d(m-d)} \Syl_{0,d}(A,B)(x),\end{equation}
which is one of the relationships between single sums that we can derive from  Sylvester's statements in his original work.

Subresultants were also introduced by Sylvester in the same article. For $f(x) = \sum_{i=0}^m f_ix^i$ with $f_m\ne 0$,   $g(x)= \sum_{i=0}^n g_ix^i$ with $g_n\ne 0$,  and $d\le \min\{m,n\}$ when $m\ne n$ or $d<m=n$,
 $$\Sres_d(f,g)(x):= \det \begin{array}{|cccccc|c}
\multicolumn{6}{c}{\scriptstyle{m+n-2d}}&\\
\cline{1-6}
f_{m} & \cdots & &\cdots &f_{d+1-(n-d-1)} &x^{n-d-1}f(x)& \\
&  \ddots & &&\vdots  & \vdots& \scriptstyle{n-d}\\
& & f_m& \dots &f_{d+1}&f(x)& \\
\cline{1-6}
g_{n} &\cdots & &\cdots  &g_{d+1-(m-d-1)}  &x^{m-d-1}g(x)&\\
&\ddots && &\vdots  &\vdots  &\scriptstyle{m-d}\\
&&g_{n} & \cdots &  g_{d+1} &g(x)&\\
\cline{1-6} \multicolumn{2}{c}{}
\end{array}.$$
This turns out to be  a polynomial in $x$ of degree bounded by $d$.

Associating to $A$ and $B$ the monic polynomials $f(x):=(x-\alpha_1)\cdots (x-\alpha_m)$ and $g(x):=(x-\beta_1)\cdots(x-\beta_n)$, one has
for instance
\begin{align*} &\Syl_{0,0}(A,B)  = \RR(A,B) =\Res(f,g),\\
& \Syl_{m,0}(A,B) = \RR(x,A)=f \ (= \Sres_m(f,g)    \mbox{ if } m<n),\\
& \Syl_{0,n}(A,B) = \RR(x,B)=g  \ (= \Sres_{n}(f,g)    \mbox{ if } n<m),\\
 & \Syl_{m,n}(A,B) = \RR(A,B)\RR(x,A)\RR(x,B)=\Res(f,g) \,f \,g ,\end{align*}
where $\Res(f,g):=\Sres_0(f,g) $ is the resultant of $f$ and $g$, which is well-known to satisfy the Poisson formula $\Res(f,g)= \prod_{\alpha\in A}g(\alpha)=\RR(A,B)$.

Sylvester also mentions in his article \cite{sylv} the link between double sums and subresultants, and many other expressions for $\Syl_{p,q}(A,B) $ depending on the values of $p$ and $q$, see also \cite{LP03,DHKS07,RS11}. The full description of $\Syl_{p,q}(A,B) $ for all possible cases of $0\le p\le m$, $0\le q\le n$ is given as follows:

\begin{theorem}\label{fulldescription} (See \cite[Main Th. 1]{DHKS09} and also \cite[Th. 1]{KS12}.)\\
Let $1\le m\le n,
\;\; 0\le p\le m, \;\;0\le q\le n,\;$  and set $d:=p+q$ and $k :=
m+n-d-1$. Then
 {
 \small
$$\operatorname{Sylv}_{p,q}(A,B)=\left\{
\begin{array}{lll}
(-1)^{p(m-d)}{d\choose p}
\Sres_{d}(f,g)&\mbox{ \ for \ }& 0\le d<  m \;\mbox{ or } \;m=d<n\\[1mm] 0&\mbox{ \ for \ }& m<d<n-1\\[1mm]
(-1)^{(p+1)(m+n-1)}{m\choose p} f&\mbox{ \ for \ }& m<d=n-1\\[1mm]
(-1)^{\sigma} \Big( {k\choose m-p}F_k(f,g) \ f- {k\choose
n-q}G_k(f,g) \ g\Big)& \\
= (-1)^{\sigma+1} \Big(  {k\choose
n-q}\Sres_k(f,g)- {k+1\choose m-p}F_k(f,g) \ f\Big) &\mbox{ \ for \ }&n\le d\le  m+n-1\\[1mm]
\Res(f,g)\ f\, g&\mbox{ \ for \ } &d=m+n.
\end{array}
\right. $$} where $\sigma  := (d-m)(n-q)+d-n-1$ and
$F_k(f,g) $ and $G_k(f,g)$ are the polynomial coefficients of $f $ and $g $ in the Bezout identity
\begin{equation}\label{Bezout}\Sres_k(f,g) =F_k(f,g) \,f + G_k(f,g)\,g,\end{equation} given by the determinantal expressions
{\small
\begin{align*}
 F_k(f,g)(x) &
:=  \det \begin{array}{|cccccc|c}
\multicolumn{6}{c}{\scriptstyle{m+n-2k}}&\\
\cline{1-6}
f_{m} & \cdots & &\cdots &f_{k+1-(n-k-1)} &x^{n-k-1}& \\
&  \ddots & &&\vdots  & \vdots& \scriptstyle{n-k}\\
& & f_m& \dots &f_{k+1}&1& \\
\cline{1-6}
g_{n} &\cdots & &\cdots  &g_{k+1-(m-k-1)}  &0&\\
&\ddots && &\vdots  &\vdots  &\scriptstyle{m-k}\\
&&g_{n} & \cdots &  g_{k+1} &0&\\
\cline{1-6} \multicolumn{2}{c}{}
\end{array}\\
G_k(f,g)(x) &
 := \det%
\begin{array}{|cccccc|c}
\multicolumn{6}{c}{\scriptstyle{m+n-2k}}&\\
\cline{1-6}
f_{m} & \cdots & &\cdots &f_{k+1-(n-k-1)} &0& \\
&  \ddots & &&\vdots  & \vdots& \scriptstyle{n-k}\\
& & f_m& \dots &f_{k+1}&0& \\
\cline{1-6}
g_{n} &\cdots & &\cdots  &g_{k+1-(m-k-1)}  &x^{m-k-1}&\\
&\ddots && &\vdots  &\vdots  &\scriptstyle{m-k}\\
&&g_{n} & \cdots &  g_{k+1} &1&\\
\cline{1-6} \multicolumn{2}{c}{}
\end{array}.
\end{align*}}

\end{theorem}

We note that Theorem \ref{fulldescription}, even if stated for $m\le n$,  indeed gives  the full description of $\Syl_{p,q}(A,B)$ in terms of $\Sres_{d}(f,g)$ and $F_k(f,g)$, $G_k(f,g)$ for any value of $m$ and $n$ because of the symmetries
\begin{align}\label{symmetry}
&\Syl_{p,q}(A,B)=(-1)^{pq+(m-p)(n-q)}\Syl_{q,p}(B,A) \\
&\Sres_d(f,g)=(-1)^{(m-d)(n-d)}\Sres_d(g,f). \nonumber
\end{align}

This theorem implies in particular Identity~\eqref{relsingle}. In fact many authors worked out the relationship between    single sums $\Syl_{d,0}(A,B)$ and   subresultants $\Sres_d(f,g)$ in the case $d\le \min\{m,n\}$ when $m\ne n$ or $d<m=n$, but all descriptions involving double sums or the other cases of $p$ and $q$ were much harder and unnatural to obtain.
In \cite{DHKS09}, Theorem \ref{fulldescription} was obtained as the determinant of an intricate matrix expression describing $\Syl_{p,q}(A,B)$ while in
\cite{KS12} it was obtained by a careful induction  from some extremal cases, knowing of course in advance what one wants to show.
Here we show that Theorem \ref{fulldescription} is in fact a natural consequence of interpreting single and double sums as specific instances of  symmetric multivariate  Lagrange interpolation and the Exchange Lemma. On one hand, symmetric interpolation yields very easily the identity between single sums and subresultants (answering thus the question of a referee of \cite{DHKS09} who asked  whether this could be obtained using specialization instead of linear algebra). This is because Sylvester single sums can be viewed as generalizations of Lagrange interpolation formulas, as it is for the case $p=0$, $q=n-1$ explained below. This is in fact the way we recovered Chen and Louck's symmetric interpolation result \cite{CL96}, as we were unaware of its existence.
We also note that Lascoux in \cite[Section 3.5]{las} mentions the possibility of using Lagrange interpolation techniques to prove identities for  Sylvester single sums, without developing it. On the other hand, the Exchange Lemma shows that in fact a natural relationship exists  between single and double sums (Propositions~\ref{dsmall1} and \ref{dbig}  below), therefore yielding all  remaining expressions in Theorem~\ref{fulldescription}.

Section \ref{section1} below treats the particular case of the Sylvester single sum $\Syl_{0,d}$, which not only motivates the use of the interpolation technique we are referring to, but ends up being a key case for the general case treated in Section \ref{section2}. We emphasize in Section \ref{section1} the very simple symmetric Lagrange interpolation Proposition \ref{Lagr}, see \cite[Th.2.1]{CL96},  which is the basis for all our development, and
 Proposition \ref{matrixforms}, which allows to make the link between the single sum and the subresultant. In Section \ref{section2} we stress again the crucial Exchange Lemma \ref{exchange}, also obtained as a consequence of the symmetric Lagrange interpolation  Proposition \ref{Lagr}, that seems to be novel  and allows to express all cases of Sylvester's double sums $\Syl_{p,q}(A,B)$ in terms of the particular cases $\Syl_{0,d}(A,B)$ and $\Syl_{m,d-m}(A,B)$, where $d:=p+q$. Theorem \ref{fulldescription} is then obtained as a consequence of Corollaries \ref{dsmall} and \ref{dbig}. In addition,  in Corollary \ref{FkGk} we obtain expressions in roots for the polynomials $F_k(f,g)$ and $G_k(f,g)$ in Bezout identity \eqref{Bezout} below. Finally,  we show in Corollary \ref{Schur}
 that  $F_k(f,g)$ and $G_k(f,g)$ are symmetric polynomials in $A\cup\{x\}$ and $B\cup \{x\}$, respectively.

\section{Sylvester's single sums  and  symmetric Lagrange interpolation}\label{section1}
We keep the following notation for the whole paper:
\begin{align*}A&=(\alpha_1,\dots, \alpha_m),\quad f=(x-\alpha_1)\cdots (x-\alpha_m)=\sum_{i=0}^m f_ix^i,\ \mbox{where} \ f_m=1,\\
B&=(\beta_1,\dots, \beta_n),\quad g=(x-\beta_1)\cdots (x-\beta_n)=\sum_{i=0}^n g_ix^i ,\ \mbox{where} \  g_n=1.\end{align*}
The double sum  expression specializes when $p=0$ and $q=d\le n$ to the following {\em single sum} expression:

\begin{align*} \Syl_{0,d}(A,B)(x) \ & =\ \sum_{B^{\prime}\subset
B,|B^{\prime}|=d}
\mathcal{R}(A,B\backslash
B^{\prime})\,\frac{\mathcal{R}(x,B^{\prime})}{\mathcal{R}(B^{\prime},B\backslash B^{\prime})} \\&= (-1)^{(m-d)(n-d)} \sum_{{B^{\prime }\subset
B,|B^{\prime}|=d}}\mathcal{R}(B\backslash B',A)\,\frac{\mathcal{R}(x,B^{\prime })
}{\mathcal{R}(B\backslash
B^{\prime }, B^{\prime})}\\ & =(-1)^{(m-d)(n-d)} \sum_{{B^{\prime }\subset
B,|B^{\prime}|=d}}\Big(\prod_{\beta\notin B'}f(\beta)\Big) \,\frac{
\prod_{\beta\in B'}(x-\beta)}{\mathcal{R}(B\backslash
B^{\prime }, B^{\prime})}.\end{align*}
In this section  we investigate the relationship between this single sum expression and a specific multivariate symmetric Lagrange  interpolation instance. As we will see in next section, the symmetric interpolation tools that we describe here  for the single sum expressions will be crucial to tackle the claims about Sylvester's  double sums.


As a motivation for what follows,  we note that when  $d=n-1$ we get
 \[\Syl_{0,n-1}(A,B)(x) = (-1)^{m+n-1}\sum_{1\le i\le n}f(\beta_i) \,\frac{\prod_{j\ne i}
(x-\beta_j)}{\prod_{j\ne i}(\beta_i-\beta_j)},\]
where it is well-known that the set \[\Big\{\frac{\prod_{j\ne i}
(x-\beta_j)}{\prod_{j\ne i}(\beta_i-\beta_j)}\,;\, 1\le i\le n\Big\}\] forms the so-called Lagrange basis, used to define the solution of the Lagrange interpolation problem, that is   the unique  polynomial of degree $\le n-1$ that takes given  values at the $n$ nodes $\beta_1,\dots,\beta_n$. Therefore, $\Syl_{0,n-1}(A,B)(x)$ equals  the unique polynomial $h_{n-1}$ of degree $\le n-1$ which satisfies the $n$ conditions $h_{n-1}(\beta_1)=(-1)^{m+n-1}f(\beta_1),\dots,h_{n-1}(\beta_n)=(-1)^{m+n-1}f(\beta_n)$.
In particular, when $m<n$,  $\Syl_{0,n-1}(A,B)(x)=(-1)^{m+n-1} f(x)$.

When trying to generalize this to the case when $d<n-1$, the difficulty  is that  the polynomials
\[ \big\{\RR(x,B')=\prod_{\beta\in B'}(x-\beta)\,;\, B'\subset B, |B'|=d \big\}\]
are linearly dependent in the vector space $K_d[x]$ of polynomials of degree bounded by $d$, since there are $\binom{n}{d}>d+1$ of them (here $K=\Q(A,B)$, a field containing $A$ and $B$).

This can be fixed for $0\le d\le n-1$ by considering a  symmetric multivariate  interpolation problem.

\subsection{Symmetric Lagrange interpolation} \label{symint}
\begin{notation} We denote by  $S_{(\ell,d)}$  the $K$-vector space of all {\em symmetric} polynomials $h$ in $\ell$ variables $x_1,\dots,x_\ell$ of multidegree bounded by $(d,\dots,d)$, i.e. such that $\deg_{x_i}(h)\le d$ for  $1\le i\le \ell$ (with no specified bound for the total degree of $h$).\end{notation}

\begin{lemma} \label{dim} $\dim_K(S_{(\ell,d)})=\binom{\ell+d}{d}$. \end{lemma}

\begin{proof} It is well-known by the fundamental theorem of elementary symmetric polynomials that the symmetric polynomials in $\ell$-variables are generated as an algebra by the elementary symmetric polynomials $$e_1(x_1,\dots,x_\ell)=x_1+\cdots+x_\ell, \dots, e_\ell(x_1,\dots,x_\ell)=x_1\cdots x_\ell,$$  all  homogeneous  of degree $1$ in each variable $x_i$. Therefore, each symmetric polynomial of multidegree bounded by $(d,\dots,d)$ can be uniquely expressed as $h=\sum_{\bfa}c_\bfa e_1^{a_1}\cdots e_\ell^{a_\ell}$ satisfying $|\bfa|:=a_1+\cdots +a_\ell\le d$. Thus it corresponds to a polynomial in $e_1,\dots,e_\ell$ of total-degree bounded by $d$: the space of such polynomials   has dimension $\binom{\ell+d}{d}$.
\end{proof}

We note that when $d<n$ and $\ell:=n-d$, then $\dim_K(S_{(n-d,d)})=\binom{n}{d}$.\\

Next proposition was proved in \cite[Th.2.1]{CL96}, but we include its proof here for sake of completeness. It shows that the set $\{\frac{\RR(\bfx,B')}{\RR(B\backslash B',B')},\ B'\subset B, |B'|=d\}$ is the Lagrange interpolation basis for  all symmetric polynomials in $n-d$ variables $X=\{x_1,\dots,x_{n-d}\}$ of multidegree bounded by $(d,\dots,d)$.

\begin{proposition} \label{Lagr} Set $0\le d\le n-1$ and $\bfx:=(x_1,\dots,x_{n-d})$.  Given  $B=\{\beta_1,\dots,\beta_n\}$, the  set  \[ {\mathcal B}: =\big\{\RR(\bfx,B')\,;\, B'\subset B, |B'|=d \big\}\,\subset S_{(n-d,d)}\]
is a basis of $S_{(n-d,d)}$.\\
Moreover, any polynomial $h(\bfx)\in S_{(n-d,d)}$ satisfies
$$h(\bfx) = \sum_{B'\subset B, |B'|=d} h(B\backslash B') \frac{\RR(\bfx,B')}{{\mathcal{R}(B\backslash
B^{\prime }, B^{\prime})}}$$
where $h(B\backslash B'):=h(\beta_{i_1},\dots,\beta_{i_{n-d}})$ for $B\backslash B'=\{\beta_{i_1},\dots,\beta_{i_{n-d}}\}$.\end{proposition}

\begin{proof} Since there are exactly $\binom{n}{d}=\dim_K(S_{(n-d,d)})$ elements in $\mathcal B$, it is enough to prove that all the elements are linearly independent. Suppose
$$\sum_{B'\subset B, |B'|=d} c_{B'} \RR(\bfx,B') =0.$$
For each subset $\{\beta_{i_1},\dots,\beta_{i_{n-d}}\}\subset B$, if we evaluate the left hand side at  $x_1=\beta_{i_1},\dots, x_{n-d}=
\beta_{i_{n-d}}$, every term in the sum vanishes except for $B':=B\backslash\{ \beta_{i_1},\dots,
\beta_{i_{n-d}}\}$, where it gives $c_{B'}\mathcal{R}(B\backslash
B^{\prime }, B^{\prime})$. Since $\mathcal{R}(B\backslash
B^{\prime }, B^{\prime})\neq 0$, we get that $c_{B'}=0$, proving linear independence. The second claim follows from the fact that $h(\bfx)\in S_{(n-d,d)}$ is uniquely expressed in the basis ${\mathcal B}$, and its coordinates are uniquely defined by all evaluations at $\{\beta_{i_1},\dots,\beta_{i_{n-d}}\}\subset B$.
\end{proof}

\subsection{Sylvester single sum}\label{singlesum}
\begin{notation}\label{hd} Set $0\le d\le n-1$ and $\bfx:=\{x_1,\dots,x_{n-d}\}$. We define
$$\mSyl_{0,d}(A,B)(\bfx):=(-1)^{(m-d)(n-d)}\sum_{B'\subset B, |B'|=d} \Big(\prod_{\beta\notin B'} f(\beta) \Big) \frac{\RR(\bfx,B')}{\mathcal{R}(B\backslash
B^{\prime }, B^{\prime})}.$$
\end{notation}

\begin{observation} By Proposition \ref{Lagr}, $\mSyl_{0,d}(A,B)(\bfx)\in S_{(n-d,d)}$ is the unique  polynomial in $S_{(n-d,d)}$ satisfying the $\binom{n}{d}$ conditions \[\mSyl_{0,d}(A,B)(B\backslash B')=(-1)^{(m-d)(n-d)} \prod_{\beta\in B\backslash B'} f(\beta) \mbox{ for all  } B'\subset B, |B'|=d.\]
\end{observation}

 In particular, \begin{equation}\label{degf<}\mSyl_{0,d}(A,B)(\bfx)=(-1)^{(m-d)(n-d)}f(x_1)\cdots f(x_{n-d})\ \mbox{ when } \deg(f)\le d.\end{equation}
The choice  of the notation $\mSyl_{0,d}(A,B)(X)$ for this polynomial stands for {\em multivariate Sylvester's sum}: the polynomial coincides with  $\Syl_{0,d}(A,B)(x)$ when $X=\{x\}$, i.e. $d=n-1$, or the latter is a coefficient of the former when there are two variables or more, i.e. $d<n-1$:

\begin{remark}\label{coeff} Set $0\le d\le n-1$ and $\bfx:=(x_1,\dots,x_{n-d})$. Then
$$\Syl_{0,d}(A,B)(x_{n-d})= \left\{\begin{array}{lcl} \coeff_{x_1^d\cdots x_{n-d-1}^d} \big( \mSyl_{0,d}(A,B)(\bfx)\big)&\mbox{for} & 0\le d<n-1 \\
\mSyl_{0,d}(A,B)(x_{n-d})& \mbox{for} & d=n-1.\end{array}\right.$$
Here $\coeff_{x_1^d\cdots x_{n-d-1}^d}$ denotes   the coefficient in $K[x_{n-d}]$ of the monomial $x_1^d\cdots x_{n-d-1}^d$
of $\mSyl_{0,d}(A,B)(X)$.
\end{remark}

Together with \eqref{degf<}, this immediately implies:
\begin{corollary}\label{cases} Set $0\le d\le n-1$. If $m\le d$ then
$$\Syl_{0,d}(A,B)= \left\{\begin{array}{lcl} 0& \mbox{for} & m<d<n-1\\(-1)^{(m-d)(n-d)}\,f&\mbox{for} & m<d=n-1 \ \mbox{or} \ m=d\le n-1 .\end{array}\right.$$
\end{corollary}

Next we show a matrix formulation for the polynomial $\mSyl_{0,d}(A,B)$ that will allow to recover the value of $\Syl_{0,d}(A,B)(x)$ for the remaining case $d\le m $.
We need to introduce the following notations for the Vandermonde matrix.

\begin{notation} \label{vander}
Let $\bfx=(x_1,\dots,x_k)$ be a $k$-tuple of (distinct) indeterminates or elements. We denote
 the (shifted) Vandermonde matrix of size $\ell \times k$ corresponding to $\bfx$ by
\begin{eqnarray*}\label{Vandermonde}
V_\ell(\bfx) := \begin{array}{|ccc|c}
\multicolumn{3}{c}{ \scriptstyle{k}}&\\
\cline{1-3}\
x_1^{\ell-1}& \dots &x_k^{\ell-1}&\\
\vdots & & \vdots & \scriptstyle{\ell}\\
1& \dots &1&\\
\cline{1-3} \multicolumn{2}{c}{}
\end{array}.\end{eqnarray*}
When $\ell=k$ we simply write $V(\bfx):=V_\ell(\bfx)$, and we recall that $\det(V(\bfx))= \prod_{1\le i<j\le k}(x_i-x_j) $.
\end{notation}

\begin{proposition} \label{matrixforms} Set $0\le d\le \min\{n-1, m\}$ and $\bfx:=\{x_1,\dots,x_{n-d}\}$.  The polynomial $\mSyl_{0,d}(A,B)(\bfx)$ of Notation \ref{hd} satisfies the following determinantal expression:

{\small
\begin{align*}\mSyl_{0,d}&(A,B)(X)=\\ &  \frac{ \det%
\begin{array}{|ccccc|ccc|c}
\multicolumn{5}{c}{\scriptstyle{m-d}}&\multicolumn{3}{c}{\scriptstyle{n-d}}\\
\cline{1-8}
f_{m} & \cdots & &\cdots &f_{d+1} &x_1^{n-d-1}f(x_1)& \cdots& x_{n-d}^{n-d-1}f(x_{n-d})&\\
& \ddots && &\vdots  & \vdots &  & \vdots &\scriptstyle{n-d}\\
& & f_m& \dots &f_{n}&f(x_1)& \cdots& f(x_{n-d})& \\
\cline{1-8}
g_{n} &\cdots & &\cdots  &g_{n-(m-d-1)}  &x_1^{m-d-1}g(x_1)& \cdots & x_{n-d}^{m-d-1}g(x_{n-d})&\\
& &\ddots & &\vdots  &\vdots & & \vdots &\scriptstyle{m-d}\\
&&  &   &  g_{n} &g(x_1)&\cdots  & g(x_{n-d})&\\
\cline{1-8} \multicolumn{2}{c}{}
\end{array}}{\det(V(\bfx))}
.\end{align*}}

\end{proposition}

\begin{proof}
In view of the definition of $\mSyl_{0,d}(A,B)$, we only  need to check that the expression given in the right-hand side of the equality is a polynomial, which is symmetric,   of degree at most $d$ in each variable $x_k$, and that when specializing the expression into $(\beta_{i_1},\dots,\beta_{i_{n-d}})$ it gives $(-1)^{(m-d)(n-d)}f(\beta_{i_1})\cdots f(\beta_{i_{n-d}})$.

It is a polynomial because the denominator divides the numerator: for each $j>i$   the term $x_j-x_i$ of $\det(V(x_1,\dots,x_{n-d}))$ divides the numerator (letting $x_i=x_j$ yields the vanishing of the above determinant).
This polynomial  is symmetric because   permuting $x_i$ with $x_j$ changes the sign of the determinants both in the numerator and in the denominator.

 Let us show the degree bound  for $x_1$. We denote by $C(j)$ the column $j$ of the matrix in the numerator. Then, performing the change
$$C(m-d+1) \mapsto C(m-d+1) -x_1^{m+n-d-1}C(1)-\cdots  - x_1^nC(m-d)=:C'(m-d+1)$$
does not change the determinant. However, we have
$$\left\{\begin{array}{l}
C'(m-d+1)_1 = f_d x_1^{n-1} + \cdots \\
\qquad  \vdots   \\
C'(m-d+1)_{n-d} =  f_{n-1}x_1^{n-1}+\cdots\\
C'(m-d+1)_{n-d+1} =  g_{n-(m-d-1)-1}x_1^{n-1}+\cdots \\
\qquad \vdots   \\
C'(n-d+1)_{m+n-2d} =  g_{n-1}x_1^{n-1} +\cdots
\end{array}\right.
$$
Therefore, the degree in $x_1$ of the top determinant is bounded by $n-1$  while the degree in $x_1$ of $\det(V(x_1,\dots,x_{n-d}))$ is exactly $n-d-1$, which implies that the degree in $x_1$ of the quotient is bounded by $n-1-(n-d-1)=d$.

We then evaluate the right-hand side expression into a $(n-d)$-tuple $(\beta_{i_1},\dots,\beta_{i_{n-d}})$ for fixed $1\le i_1<\cdots <i_{n-d}\le n$. It is clear that the top determinant equals
$$(-1)^{(m-d)(n-d)} \det(V(\beta_{i_1},\dots,\beta_{i_{n-d}}))f(\beta_{i_1})\cdots f(\beta_{i_{n-d}}), $$
while the bottom determinant equals $\det(V(\beta_{i_1},\dots,\beta_{i_{n-d}}))$.
This concludes the proof.

\end{proof}

Note that Proposition \ref{matrixforms} is  very similar in spirit to the matrix definition of the subresultant: actually they coincide when $d=n-1$. This inspires the following result.

\begin{proposition}\label{ss=sub} Set $0\le d\le n-1$. For  $d\le m$, one has $$\Syl_{0,d}(A,B) =\Sres_d(f,g) .$$
\end{proposition}
\begin{proof}
We  denote by  $S(\bfx)$ the polynomial in the numerator of the expression for $\mSyl_{0,d}(A,B)(\bfx)$ in the right-hand side of Lemma \ref{matrixforms}, and by $c_d(x_{n-d})\in K[x_{n-d}]$ the coefficient of $x_{1}^d \cdots x_{n-d-1}^d$ in $\mSyl_{0,d}(A,B)(\bfx)$, that we want to show equals   $\Sres_d(f,g)(x_{n-d})$ according to Remark \ref{coeff}.

Since $\mSyl_{0,d}(A,B)(\bfx)= S(\bfx)/\det(V(\bfx))$, we get \begin{align*} S(\bfx)& =  \mSyl_{0,d}(A,B)(\bfx)\det(V(\bfx)) \\ & = (c_d(x_{n-d}) \, x_{1}^d \cdots x_{n-d-1}^d + \cdots)(x_1^{n-d-1}x_2^{n-d-2} \cdots  x_{n-d-1}+ \cdots )\\
 & =   c_d(x_{n-d}) \, x_1^{n-1}x_2^{n-2} \cdots  x_{n-d-1}^{d+1} + \cdots\end{align*}
Therefore,
 $$c_d(x_{n-d})=\coeff_{x_1^{n-1}x_2^{n-2} \cdots x_{n-d-1}^{d+1}}(S(\bfx)).$$
It is clear that the coefficient of $x_1^{n-1} \cdots  x_{n-d-1}^{d+1}$ in the determinant is obtained, by multilinearity, from the coefficients when the column with $x_1$ has all its exponents equal to $n-1$, the column with $x_2$ has all its exponents equal to $n-2$ up to the  column with $x_{n-d-1}$ has all its exponents equal  to $d+1$, that is $\coeff_{x_1^{n-1}\cdots  x_{n-d-1}^{d+1}}(S(\bfx))$ equals

{\scriptsize
$$
\det\begin{array}{|ccccc|cccc|c}
\multicolumn{5}{c}{\scriptstyle{m-d}}&\multicolumn{4}{c}{\scriptstyle{n-d}}\\
\cline{1-9}
f_{m} & \cdots & &\cdots &f_{d+1} & f_d &\cdots& f_{d+1-(n-d-1)}&x_{n-d}^{n-d-1}f(x_{n-d})&\\
& \ddots && &\vdots  & \vdots & \vdots& & \vdots &\scriptstyle{n-d}\\
& & f_m& \dots &f_{n}& f_{n-1}&\cdots& f_{d+1}& f(x_{n-d})&\\
\cline{1-9}
g_{n} &\cdots & &\cdots  &g_{n-(m-d-1)}  & g_{n-(m-d)}&\cdots & g_{d+1- (m-d-1)}&x_{n-d}^{m-d-1}g(x_{n-d})&\\
& &\ddots & &\vdots  &\vdots & \vdots&& \vdots &\scriptstyle{m-d}\\
&&  &   &  g_{n} &g_{n-1}&\cdots  & g_{d+1}&g(x_{n-d})&\\
\cline{1-9} \multicolumn{2}{c}{}
\end{array}$$}
Thus   $$\coeff_{x_1^{n-1}x_2^{n-2} \cdots x_{n-d-1}^{d+1}}(S(\bfx))=\Sres_d(f,g)(x_{n-d}),$$ which implies
$c_d(x_{n-d})=\Sres_d(f,g)(x_{n-d})$ as desired.
\end{proof}

Putting together the information of Corollary \ref{cases}, Proposition \ref{ss=sub}  and the value of $\Syl_{0,d}$ for $d=n$, this interpolation technique therefore allowed us to recover  very naturally the full description of Sylvester's single sums $\Syl_{0,d}$ for any $0\le d\le n$ and any $m$:

\begin{equation*}\label{0d} \Syl_{0,d}(A,B) =\left\{\begin{array}{lcl}
\Sres_d(f,g) & \mbox{for} & 0\le d\le \{ m-1,n\} \mbox{ or } d=m<n\\
0& \mbox{for} & m<d<n-1\\
(-1)^{m+n-1}f & \mbox{for} & m<d=n-1\\
g & \mbox{for} &  m\le d=n.
\end{array}\right. \end{equation*}

\section{Sylvester's double sums}\label{section2}

We treat now the case of the  general double sum expression, defined for $0\le p\le m$, $0\le q\le n$. Below we show   how all  cases reduce to the specific instances of $\Syl_{0,d}$ and $\Syl_{m,d-m}$. \\

The whole section flows from the following Exchange Lemma and its Corollary, which we could prove thanks to the interpolation Proposition \ref{Lagr} on symmetric polynomials.

\begin{lemma}\label{exchange} Let $A$ be any set with $|A|=m$, and set $0\le p\le m$ and $\bfx=\{x_1,\dots,x_{m-p}\}$.  Let $B$ be any set with $|B|\ge p$. Then

\[\sum_{{A^{\prime}\subset A,|A^{\prime}|=p}}\mathcal{R}(A\backslash A^{\prime},B)\frac{\mathcal{R}(\bfx,A^{\prime})}{\mathcal{R}(A\backslash A^{\prime},A^{\prime})}= \sum_{{B^{\prime}\subset B,|B^{\prime}|=p}}\mathcal{R}(A,B\backslash B^{\prime})\frac{\mathcal{R}(\bfx,B^{\prime})}{\mathcal{R}(B^{\prime},B\backslash B^{\prime})} .\]
\end{lemma}

\begin{proof} We observe that Proposition \ref{Lagr} implies that the polynomial $h(\bfx)\in s_\lambda{(m-p,p)}$ on the left-hand side is the only symmetric polynomial satisfying the $\binom{m}{p}$ conditions
$h(A\backslash A^{\prime})=\mathcal{R}(A\backslash A^{\prime},B)$. Since the polynomial on the right-hand side also belongs to $S_{(m-p,p)}$, it suffices to show that it satisfies the same specialization properties, i.e. that
{ \[\sum_{{B^{\prime}\subset B,|B^{\prime}|=p}}\mathcal{R}(A,B\backslash B^{\prime})\frac{\mathcal{R}(A\backslash A^{\prime},B^{\prime})}{\mathcal{R}(B^{\prime},B\backslash B^{\prime})}=\mathcal{R}(A\backslash A^{\prime},B), \ \forall A'\subset A, |A'|=p.\]}
But
{\small \begin{align*}\sum_{{B^{\prime}\subset B,|B^{\prime}|=p}}\mathcal{R}(A,B\backslash B^{\prime})\frac{\mathcal{R}(A\backslash A^{\prime},B^{\prime})}{\mathcal{R}(B^{\prime},B\backslash B^{\prime})}& =
 \sum_{{B^{\prime}\subset B,|B^{\prime}|=p}}\mathcal{R}(A',B\backslash B^{\prime})\mathcal{R}(A\backslash A',B\backslash B^{\prime})\frac{\mathcal{R}(A\backslash A^{\prime},B^{\prime})}{\mathcal{R}(B^{\prime},B\backslash B^{\prime})}\\
 &=  \RR(A\backslash A',B)\sum_{{B^{\prime}\subset B,|B^{\prime}|=p}}\frac{\mathcal{R}(A',B\backslash B^{\prime})}{\mathcal{R}(B^{\prime},B\backslash B^{\prime})}.
\end{align*}}
Consider for  $\bfy=\{y_1,\dots,y_p\}$ the polynomial $$\Psi(\bfy) =\sum_{{B^{\prime}\subset B,|B^{\prime}|=p}} \frac{\mathcal{R}(\bfy,B\backslash B^{\prime})}{\mathcal{R}(B^{\prime},B\backslash B^{\prime})} \in S_{(p,|B|-p)}.$$ It is, again by Proposition \ref{Lagr},  the only polynomial in $S_{(p,|B|-p)} $ satisfying the $\binom{|B|}{p}$ conditions  $\Psi(B')=1$, $\forall B'\subset B, \, |B'|=p$, and therefore $\Psi=1$. In particular $\Psi(A')=1$, which implies the statement.
\end{proof}

\begin{corollary} \label{exchange2} Let $A$ be any set with $|A|=m$, and set $0\le p\le m$ and  $\bfx=(x_1,\dots,x_r)$, with $r\le m-p$.  Let $B$ be any set with $|B|\ge p$. Then

\[\sum_{{A^{\prime}\subset A,|A^{\prime}|=p}}\mathcal{R}(A\backslash A^{\prime},B)\frac{\mathcal{R}(\bfx,A^{\prime})}{\mathcal{R}(A\backslash A^{\prime},A^{\prime})}= \sum_{{B^{\prime}\subset B,|B^{\prime}|=p}}\mathcal{R}(A,B\backslash B^{\prime})\frac{\mathcal{R}(\bfx,B^{\prime})}{\mathcal{R}(B^{\prime},B\backslash B^{\prime})} .\]
\end{corollary}
\begin{proof} The expressions above is simply the coefficient of $x_{r+1}^p\cdots x_{m-p}^p$ of the expressions in Lemma \ref{exchange}.
\end{proof}

\subsection{The case $0\le d\le n-1$} {\ } \\

In this section we set $d:=p+q$, where $0\le p\le n$ and $0\le q\le n$,  and we assume that it satisfies $d\le n-1$. \\We first deal with the case
$d\le \min\{m-1,n-1\}$.
The Corollary \ref{exchange2} of the Exchange Lemma allows to relate $\Syl_{p,q}(A,B)$ to $\Syl_{0,p+q}(A,B)$,   simply by a careful manipulation of terms.

\begin{proposition} \label{dsmall1} Let $0\le p\le m$, $0\le q\le n$  and set $d:=p+q$. Assume   $d\le \min\{m-1, n-1\}$. Then  $$\Syl_{p,q}(A,B) = (-1)^{p(m-d)}\binom{d}{p}\Syl_{0,d}(A,B) .$$
\end{proposition}

\begin{proof}
First note that $$\Syl_{p,q}(A,B)= (-1)^{p(m-d)} \sum_{\substack{B^{\prime}\subset
B\\|B^{\prime}|=q}}\,\left(\sum_{\substack{A^{\prime
}\subset A\\|A^{\prime}|=p}}\mathcal{R}(A\backslash A^{\prime},B\backslash
B^{\prime})\frac{
\mathcal{R}(x\cup B',A^{\prime
})}{\mathcal{R}(A\backslash A^{\prime
},A^{\prime})}\right)\frac{\mathcal{R}(x,B^{\prime})}{\mathcal{R}(B^{\prime},B\backslash B^{\prime})}. $$

Applying  the  Exchange Corollary \ref{exchange2} to the coefficients inside the parenthesis of this expression for $\bfx=x\cup B'$ of size $q+1\le m-p$ and $B\backslash B'$  of size $n-q\ge p$ instead of $B$, we get
{\small
\[\sum_{\substack{A^{\prime}\subset A\\|A^{\prime}|=p}}\mathcal{R}(A\backslash A^{\prime},B\backslash B')\frac{\mathcal{R}(x\cup B',A^{\prime})}{\mathcal{R}(A\backslash A^{\prime},A^{\prime})}= \sum_{\substack{C^{\prime}\subset B\backslash B'\\|C^{\prime}|=p}}\mathcal{R}(A,B\backslash (B'\cup C^{\prime}))\frac{\mathcal{R}(x\cup B',C^{\prime})}{\mathcal{R}(C^{\prime},B\backslash (B'\cup C^{\prime}))} .\]}
Therefore,
{\small \begin{align*}\Syl_{p,q}(A,B)&= (-1)^{p(m-d)}\sum_{\substack{B^{\prime}\subset
B\\|B^{\prime}|=q}}\,\left(\sum_{\substack{C^{\prime}\subset B\backslash B'\\|C^{\prime}|=p}}\mathcal{R}(A,B\backslash (B'\cup C^{\prime}))\frac{\mathcal{R}(x\cup B',C^{\prime})}{\mathcal{R}(C^{\prime},B\backslash (B'\cup C^{\prime}))}\right)\frac{\mathcal{R}(x,B^{\prime})}{\mathcal{R}(B^{\prime},B\backslash B^{\prime})}\\
&=  (-1)^{p(m-d)} \sum_{\substack{B^{\prime}\subset
B,|B^{\prime}|=q\\C^{\prime}\subset B\backslash B', |C^{\prime}|=p }}\mathcal{R}(A,B\backslash (B'\cup C^{\prime}))\frac{\mathcal{R}( B',C^{\prime})}{\mathcal{R}(C^{\prime},B\backslash (B'\cup C^{\prime}))}\frac{\mathcal{R}(x,B^{\prime}\cup C')}{\mathcal{R}(B^{\prime},B\backslash B^{\prime})}\\
&=(-1)^{p(m-d)} \sum_{\substack{B^{\prime}\subset
B,|B^{\prime}|=q\\C^{\prime}\subset B\backslash B', |C^{\prime}|=p }}\mathcal{R}(A,B\backslash (B'\cup C^{\prime}))\frac{\mathcal{R}(x,B^{\prime}\cup C')}{\mathcal{R}(B'\cup C',B\backslash (B^{\prime}\cup C'))}.
\end{align*}}
Finally, rewriting the sum over  $B'\subset B, |B'|=q, C'\subset B\backslash B', |C'|=p$ as a sum over  $D=B'\cup C'\subset B, |D|=d; C'\subset D, |C'|=p$, we get
\begin{align*}\Syl_{p,q}(A,B)&=(-1)^{p(m-d)}
\sum_{\substack{D\subset
B,|D|=d\\C^{\prime}\subset D, |C^{\prime}|=p }}\mathcal{R}(A,B\backslash D)\frac{\mathcal{R}(x,D)}{\mathcal{R}(D,B\backslash D)}\\
&= (-1)^{p(m-d)}\binom{d}{p} \sum_{{D\subset
B,|D|=d}}\mathcal{R}(A,B\backslash D)\frac{\mathcal{R}(x,D)}{\mathcal{R}(D,B\backslash D)}\\
&= (-1)^{p(m-d)}\binom{d}{p} \Syl_{0,d}(A,B),
\end{align*}
by the definition of $\Syl_{0,d}(A,B)$.
\end{proof}
Proposition \ref{ss=sub} then  immediately implies
\begin{corollary} \label{dsmall} Let $0\le p\le m$, $0\le q\le n$  and set $d:=p+q$. Assume   $d\le \min\{m-1, n-1\}$. Then  $$\Syl_{p,q}(A,B) = (-1)^{p(m-d)}\binom{d}{p}\Sres_{d}(f,g) .$$

\end{corollary}

We now deal with the remaining case of this section, $m\le d\le n-1$, where $d:=p+q$ with $0\le p\le m$, $0\le q\le n$.
 We  express the double sum $\Syl_{p,q}(A,B)$  by means of an interpolation problem as in Section \ref{section1}.
We rewrite   $\Syl_{p,q}(A,B)$  as
{\small \begin{align*}
\Syl_{p,q}(A,B)(x) & =  (-1)^{(m-d)(n-q)} \sum_{\substack{A^{\prime
}\subset A\\|A^{\prime}|=p}}\,\left(\sum_{\substack{B^{\prime}\subset
B\\|B^{\prime}|=q}}\mathcal{R}(B\backslash
B^{\prime}, A\backslash A^{\prime})\frac{
\mathcal{R}(x\cup A',B^{\prime
})}{\mathcal{R}(B\backslash B^{\prime
},B^{\prime})}\right)\frac{\mathcal{R}(x,A^{\prime})}{\mathcal{R}(A^{\prime},A\backslash A^{\prime})}\\ & = (-1)^{(m-d)(n-q)}\sum_{{A^{\prime
}\subset A,|A^{\prime}|=p}}h_{A'}(x\cup A')\frac{\mathcal{R}(x,A^{\prime})}{\mathcal{R}(A^{\prime},A\backslash A^{\prime})},\end{align*}}
where
$$h_{A'}(x_1,x_2,\dots,x_{p+1}):=\sum_{{B^{\prime}\subset
B,|B^{\prime}|=q}}\mathcal{R}(B\backslash
B^{\prime}, A\backslash A^{\prime})\frac{
\mathcal{R}(\{x_1,\dots,x_{p+1}\},B^{\prime
})}{\mathcal{R}(B\backslash B^{\prime
},B^{\prime})},$$
 is a symmetric polynomial in $(x_1,\dots,x_{p+1})$ of multidegree  bounded by $(q,\dots,q )$.
Since by hypothesis $p+q\le n-1$, we can complete the set $(x_1,\dots,x_{p+1})$ to the set $\bfx=(x_1,\dots,x_{n-q})$.  We have the following Lemma:

\begin{lemma} \label{deggB<} Let $0\le p\le m$, $0\le q\le n$,  and set $d:=p+q$. Assume   $m\le d\le n-1$ and   define for   $\bfx=(x_1,\dots,x_{n-q})$ the polynomial
$$ H_{A'}(\bfx):=\sum_{{B^{\prime}\subset
B,|B^{\prime}|=q}}\mathcal{R}(B\backslash
B^{\prime}, A\backslash A^{\prime})\frac{
\mathcal{R}(\bfx,B^{\prime
})}{\mathcal{R}(B\backslash B^{\prime
},B^{\prime})}.$$
Then
\begin{equation*}H_{A'}(\bfx)= f_{_{A\backslash A'}}(x_1)\cdots f_{_{A\backslash A'}}(x_{n-q})\,\end{equation*}
where $f_{_{A\backslash A'}}(x):=\prod_{\alpha \in A\backslash A'}(x-\alpha)$.
\end{lemma}

\begin{proof} Clearly $H_{A'}(\bfx)$ has multidegree $(q,\ldots, q)$, thus by Proposition \ref{Lagr}, $H_{A'}$ is the only polynomial in $S_{(n-q,q)}$ satisfying the
$\binom{n}{q} $ conditions
$$H_{A'}(B\backslash B^{\prime})= \prod_{\beta\in B\backslash B'}f_{_{A\backslash A'}}(\beta) \ \mbox{ for all  }\ B'\subset B, |B'|=q,$$
which proves the claim since $\deg (f_{_{A\backslash A'}})=m-p\le q$ by the hypothesis $m\le d$.
\end{proof}

 Note the similarity of $H_{A'}$ with $h_{A'}$:  they are the same  when  $d=n-1$, or the latter is a coefficient of the former when there are two variables or more, i.e. $d<n-1$:

\begin{remark} \label{ha'} Let $0\le p\le m$, $0\le q\le n$  and set $d:=p+q$. Assume   $m\le d\le n-1$ and set $\bfx=(x_1,\dots,x_{n-q})$. Then
$$h_{A'}(x_1,x_2,\dots,x_{p+1})= \left\{\begin{array}{lcc} \coeff_{x_{p+2}^q\cdots x_{n-q}^q}\big(H_{A'}(\bfx)\big)&\mbox{for} & 0\le d<n-1 \\
H_{A'}(\bfx)& \mbox{for} & d=n-1.\end{array}\right.$$
Here $\coeff_{x_{p+2}^q\cdots x_{n-q}^q}$ denotes   the coefficient in $K[x_1,\dots,x_{p+1}]$ of the monomial $x_{p+2}^q\cdots x_{n-q}^q$
of $H_{A'}$.
\end{remark}

Together with Corollary \ref{dsmall}, Lemma \ref{deggB<} and Remark \ref{ha'} immediately imply the following full description of  the case $d\le n-1$:
\begin{corollary}\label{dn-1} Let $0\le p\le m$, $0\le q\le n$  and set $d:=p+q$. Assume  $0\le d\le n-1$. Then,
$$\Syl_{p,q}(A,B)=\left\{\begin{array}{lcl}  (-1)^{p(m-d)}\binom{d}{p} \Sres_{d}(f,g) & \mbox{for} & 0\le d\le \min\{m,n-1\} \\
0& \mbox{for} & m<d<n-1\\(-1)^{(p+1)(m+n-1)}\binom{m}{p}\,f&\mbox{for} & m<d=n-1  .\end{array}\right.$$
\end{corollary}

\begin{proof} The first case is  direct from Corollary \ref{dsmall} for $d\le m-1$ while for $d=m$,  $\Syl_{m,0}(A,B) = f=\Sres_m(f,g) $ when $m<n$. \\The two following cases are
a consequence of the definition of $\Syl_{p,q}(A,B)$ in terms of $h_{A'}(x\cup A')$, Identity \eqref{deggB<} and Remark  \ref{ha'}.  For the second case we  observe that $h_{A'}=0$ and for the third one, that
the polynomial $f_{_{A\backslash A'}}(x_1)\cdots f_{_{A\backslash A'}}(x_{p+1})$ specialized into $x\cup A'$ equals $f_{_{A\backslash A'}}(x)\RR(A',A\backslash A')$. So, for $d=n-1$,
\begin{align*}\Syl_{p,q}(A,B)&= (-1)^{(m-d)(n-q)} \sum_{{A^{\prime
}\subset A,|A^{\prime}|=p}}h_{A'}(x\cup A')\frac{\mathcal{R}(x,A^{\prime})}{\mathcal{R}(A^{\prime},A\backslash A^{\prime})}\\
&=(-1)^{(m-d)(n-q)}\sum_{{A^{\prime}\subset
A,|A^{\prime}|=p}}\RR(x,A\backslash A')\RR(A',A\backslash A') \frac{\mathcal{R}(x,A^{\prime})}{\mathcal{R}(A^{\prime},A\backslash A^{\prime})}\\
&=(-1)^{(m+n-1)(p+1)}\sum_{{A^{\prime}\subset
A,|A^{\prime}|=p}}  f(x)\ = \ (-1)^{(p+1)(m+n-1)}\binom{m}{p}\,f(x).\end{align*}
\end{proof}

\subsection{The case $\max\{m,n\}\le d\le  m+n$} {\ } \\

In the previous section we concluded the case $0\le d\le n-1$ and $m$ arbitrary. By the symmetry mentioned in Identity \eqref{symmetry},
this also concludes the case $0\le d\le m-1$ and $n$ arbitrary. Thus it only remains to consider the case $d\ge m$ and $d\ge n$, i.e.
$\max\{m,n\}\le d\le  m+n$. \\

First we observe that the case $d=m+n$, i.e. $p=m$ and $n=q$ is already solved, since $\Syl_{m,n}(A,B)=\Res(f,g)\,f\,g$, as was mentioned in the introduction.\\

In the rest of this section we restrict to the case $\max\{m,n\}\le d\le m+n-1$. We show that we can express $\Syl_{p,q}(A,B)$ in these cases   in terms of $\Syl_{0,k}(A,B)$ and $\Syl_{m,d-m}(A,B)$, where $k:=m+n-d-1$.

\begin{proposition} \label{dbig} Let $0\le p\le m$, $0\le q\le n$  and set $d:=p+q$. Assume   $\max\{m,n\}\le d\le m+n-1$.  Then
\[
\Syl_{p,q}(A,B) = (-1)^c \binom{k}{n-q} \Syl_{0,k}(A,B) +(-1)^e \binom{k+1}{m-p} \Syl_{m,d-m}(A,B) ,\]
where  $k:=m+n-d-1$, $c:=(d-m)(n-q)+d-n$ and $e:=(d-m)(q+1)$.
\end{proposition}

\begin{proof}
We rewrite, for a fixed $B'\subset B$, $|B'|=q$,
{\small \begin{eqnarray*}\sum_{\substack{A^{\prime}\subset
A\\|A^{\prime}|=p}}\mathcal{R}(B^{\prime},A^{\prime})\mathcal{R}(B\backslash
B^{\prime},A\backslash A^{\prime})\frac{\mathcal{R}(x,A^{\prime
})}{\mathcal{R}(A^{\prime},A\backslash A^{\prime
})}  =\sum_{\substack{A^{\prime}\subset
A\\|A^{\prime}|=p}}\mathcal{R}(B^{\prime}\cup x,A')\frac{\mathcal{R}(B\backslash B^{\prime},A\backslash
A^{\prime})}{\mathcal{R}(A^{\prime
},A\backslash A^{\prime})}\\
 =  (-1)^{p(q+1)}\sum_{\substack{A^{\prime\prime}\subset
A\\|A^{\prime\prime}|=m-p}}\mathcal{R}(A\backslash A^{\prime\prime},B^{\prime}\cup x)\frac{\mathcal{R}(B\backslash B^{\prime},A^{\prime\prime})}{\mathcal{R}(A\backslash A^{\prime\prime
},A^{\prime\prime})}\\
 = (-1)^{p(q+1)}\sum_{\substack{C^{\prime\prime}\subset
B'\cup\{x\}\\|C^{\prime\prime}|=m-p}}\mathcal{R}(A,(B'\cup x )\backslash C'')\frac{\mathcal{R}(B\backslash B^{\prime},C'')}{\mathcal{R}(C'',(B'\cup x )\backslash C'')},\end{eqnarray*}}
where for the last equality we used the Exchange Corollary \ref{exchange2} for $B'\cup \{ x\}$ of size $q+1\ge m-p$ since $d\ge m-1$ instead of $B$ and $\bfx=B\backslash B'$ of size $n-q\le m-(m-p)$ since $d\ge n$.  Therefore, setting $\zeta:=(m-p)(n-q)$, we get
{\small \begin{eqnarray*}
 \Syl_{p,q}(A,B)(x)=(-1)^{\zeta+pq}
\sum_{\substack{B^{\prime
}\subset B\\|B^{\prime}|=q}}\,\Big(\sum_{\substack{A^{\prime}\subset
A\\|A^{\prime}|=p}}\mathcal{R}(B^{\prime},A^{\prime})\mathcal{R}(B\backslash B^{\prime},A\backslash
A^{\prime})\frac{\mathcal{R}(x,A^{\prime
})}{\mathcal{R}(A^{\prime},A\backslash A^{\prime
})}\Big)\frac{\mathcal{R}(x,B^{\prime})}{\mathcal{R}(B^{\prime},B\backslash B^{\prime})}\\ =
(-1)^{\zeta+p}\sum_{\substack{B^{\prime
}\subset B\\|B^{\prime}|=q}}\,\Big(\sum_{\substack{C^{\prime\prime}\subset
B'\cup\{x\}\\|C^{\prime\prime}|=m-p}}\mathcal{R}(A,(B'\cup x )\backslash C'')\frac{\mathcal{R}(B\backslash B^{\prime},C'')}{\mathcal{R}(C'',(B'\cup x )\backslash C'')}\Big)\frac{\mathcal{R}(x,B^{\prime})}{\mathcal{R}(B^{\prime},B\backslash B^{\prime})}.\end{eqnarray*}}
Now we split this sum (without considering the sign $(-1)^{\zeta+p}$ for now) in two sums, according whether $x\in C''$ or $x\notin C''$. The first sum $S_1$, when $x\in C''$,  equals
{\begin{align*}
S_1&=
\sum_{\substack{B'\subset B, |B'|=q \\C'\subset
B',|C'|=m-p-1 }}\mathcal{R}(A,B'\backslash C')\frac{\mathcal{R}(B\backslash B^{\prime},C'\cup x )}{\mathcal{R}(C'\cup x , B'\backslash C')}\frac{\mathcal{R}(x,B^{\prime})}{\mathcal{R}(B^{\prime},B\backslash B^{\prime})}
\\
&= (-1)^{n-q+(n-q)q}\sum_{\substack{B'\subset B, |B'|=q \\C'\subset
B',|C'|=m-p-1 }}\mathcal{R}(A,B'\backslash  C')\frac{\mathcal{R}(x,(B\backslash B')\cup C')}{\mathcal{R}((B\backslash B^{\prime})\cup C',B^{\prime}\backslash C')}\\
&= (-1)^{n(q+1)}\sum_{\substack{B'\subset B, |B'|=q \\C'\subset
B',|C'|=m-p-1 }}\mathcal{R}(A,B\backslash ((B\backslash B')\cup  C')\frac{\mathcal{R}(x,(B\backslash B')\cup C')}{\mathcal{R}((B\backslash B^{\prime})\cup C',B\backslash ((B\backslash B')\cup  C'))}.
\end{align*} }

Finally, replacing the summation over $B'$  by a sum over  $D=(B\backslash B^{\prime})\cup C^{\prime}\subset B$, where $|D|=n-q+m-p-1=m+n-d-1=k$ (observe that
$0\le k\le \min\{m-1,n-1\} $ since $\max\{m,n\}\le d\le m+n-1$), we get

\begin{align*}S_1&= (-1)^{n(q+1)}\sum_{\substack{D\subset B, |D|=k\\C^{\prime}\subset
D,|C^{\prime}|=m-p-1}}\mathcal{R}(A,B\backslash D)\frac{\mathcal{R}(x,D)}{\mathcal{R}(D, B\backslash D)}\\
&=(-1)^{n(q+1)} \binom{k}{m-p-1}\sum_{{D\subset B, |D|=k}}\mathcal{R}(A,B\backslash D)\frac{\mathcal{R}(x,D)}{\mathcal{R}(D, B\backslash D)}\\
&= (-1)^{n(q+1)}\binom{k}{n-q}\Syl_{0,k}(A,B)(x).
\end{align*}
Let $S_2$ denote the second sum (without the sign  $(-1)^{\zeta+p}$ for now), when $x\notin C''$:

{\begin{align*}
S_2&=\sum_{\substack{B'\subset B, |B'|=q \\C''\subset
B',|C''|=m-p }} \mathcal{R}(A,(B'\cup x)\backslash C'')\frac{\mathcal{R}(B\backslash B^{\prime},C'')}{\mathcal{R}(C'', (B'\cup x)\backslash C'')}\frac{\mathcal{R}(x,B^{\prime})}{\mathcal{R}(B^{\prime},B\backslash B')}\\ & =
(-1)^\varepsilon f(x)\sum_{\substack{B'\subset B, |B'|=q \\C''\subset
B',|C''|=m-p }} \mathcal{R}(A,B'\backslash C'')\frac{\mathcal{R}(x,B^{\prime}\backslash C'' )}{\mathcal{R}(B'\backslash C'',(B\backslash B^{\prime})\cup C'')}
\\ & =
(-1)^\varepsilon f(x)\sum_{\substack{B'\subset B, |B'|=q \\C''\subset
B',|C''|=m-p }} \mathcal{R}(A,B\backslash ((B\backslash B')\cup  C''))\frac{\mathcal{R}(x,B\backslash ((B\backslash B')\cup  C'' ))}{\mathcal{R}(B\backslash((B\backslash B')\cup  C''),(B\backslash B^{\prime})\cup C'')}
\\
&= (-1)^\varepsilon f(x)\sum_{\substack{D\subset B, |D|=k+1 \\C''\subset
D,|C''|=m-p }}\mathcal{R}(A,B\backslash D)\frac{\mathcal{R}(x,B\backslash D)}{\mathcal{R}(B\backslash D,D)}\\
&= (-1)^\varepsilon \binom{k+1}{m-p} \, f(x)\sum_{{D\subset B, |D|=k+1}}\mathcal{R}(A,B\backslash D)\frac{\mathcal{R}(x,B\backslash D)}{\mathcal{R}(B\backslash D,D)}\\ &=
 (-1)^\varepsilon \binom{k+1}{m-p} \, f(x)\sum_{{D'\subset B, |D'|=d-m}}\mathcal{R}(A,D')\frac{\mathcal{R}(x,D')}{\mathcal{R}(D',B\backslash D'))}\\
 &=  (-1)^\varepsilon \binom{k+1}{m-p} \Syl_{m,d-m}(A,B),
\end{align*} }
where $\varepsilon:=m+ m-p+(m-p)(n-q) + (m-p)(q+p-m) \equiv m + n(m-p) \pmod 2$, and $D= (B\backslash B')\cup C''$ with $|D|=m+n-d-1+1=k+1$ and $D'= B\backslash D$ with $|D'|=n-(k+1)=d-m$,
by the definition of $\Syl_{m,d-m}$.

Finally we add up  the  signs
\begin{align*}c:&=\zeta+p+ n(q+1)\equiv (d-m)(n-q)+d-n \pmod 2\\
e:&= \zeta+p+ \varepsilon = (m-p)(n-q)+m+n(m-p)\equiv  (d-m)(q+1)\pmod 2
\end{align*}
to get the expression in the claim.
\end{proof}

We end up considering the only remaining case,  $\Syl_{m,d-m}(A,B)$ for  $\max\{m,n\} \le d\le m+n-1$.
Remember that
\begin{align*}
 \Syl_{m,d-m}(A,B)(x)&=\sum_{B^{\prime}\subset
B, |B^{\prime}|=d-m}\mathcal{R}(A,B^{\prime})\,
\,\frac{\mathcal{R}(x,A)\,\mathcal{R}(x,B^{\prime})}{\mathcal{R}(B^{\prime},B\backslash B^{\prime})}\\ & =  f(x)\,\sum_{B^{\prime}\subset
B, |B^{\prime}|=d-m}\mathcal{R}(A,B^{\prime})\,
\,\frac{\mathcal{R}(x,B^{\prime})}{\mathcal{R}(B^{\prime},B\backslash B^{\prime})} ,\end{align*}
which is a polynomial of degree bounded by $m+(d-m)=d$.\\
Since $\max\{m,n\}\le d\le m+n-1$,   $k:=m+n-d-1$ satisfies $k\le \min\{m-1,n-1\}$. Thus, for these values of $k$ we have  the Bezout Identity \eqref{Bezout} $$\Sres_k(f,g)=F_k(f,g) f + G_k(f,g) g.$$
Here $F_k(f,g)$ is a polynomial of degree bounded by $n-k-1=d-m$.

\begin{proposition} \label{mq} Let $\max\{m,n\} \le d\le m+n-1$ and set $k:=m+n-d-1$. Then
$$\Syl_{m,d-m}(A,B)= (-1)^{(d-m)n+m+n-1} F_k(f,g)\,f.$$
\end{proposition}

\begin{proof}Since $\Syl_{m,d-m}(A,B)$ and $F_k(f,g)\,f$ are both polynomials of degree bounded by $d\le m+n-1$, it suffices to show the equality by interpolating them on the indeterminates of  $A\cup B$.
It is clear that they both vanish --thus coincide-- on any $\alpha\in A$. So it remains to evaluate both polynomials on $\beta \in B$. We have
\begin{align*}
 \Syl_{m,d-m}(A,B)(\beta)&=  f(\beta)\,\sum_{B^{\prime}\subset
B, |B^{\prime}|=d-m, \beta\notin B'}\mathcal{R}(A,B^{\prime})\,
\,\frac{\mathcal{R}(\beta,B^{\prime})}{\mathcal{R}(B^{\prime},B\backslash B^{\prime})} \\
& =  f(\beta)\,\sum_{B^{\prime}\subset
B\backslash \{\beta\}, |B^{\prime}|=d-m}\mathcal{R}(A,B^{\prime})\,
\,\frac{\mathcal{R}(\beta,B^{\prime})}{\RR( B',\beta)\mathcal{R}(B^{\prime},B\backslash (\beta \cup  B^{\prime}))} \\ &=
(-1)^{d-m} f(\beta)\,\sum_{B^{\prime}\subset
B\backslash \{\beta\}, |B^{\prime}|=d-m}
\,\frac{\mathcal{R}(A,B^{\prime})}{\mathcal{R}(B^{\prime},B\backslash (\beta \cup  B^{\prime}))}.\end{align*}
On the other hand, since $g(\beta)=0$,
$$(F_k(f,g)\,f)(\beta) = (F_k(f,g)\,f+G_k(f,g)\,g)(\beta)=\Sres_k(f,g)(\beta)= \Syl_{0,k}(A,B)(\beta)$$
by Proposition \ref{ss=sub}. Therefore we need to compute $\Syl_{0,k}(A,B)(\beta)$.
But \begin{align*}
 \Syl_{0,k}(A,B)(\beta)&=\sum_{D\subset
B, |D|=k}\mathcal{R}(A,B\backslash D)\,
\,\frac{\mathcal{R}(\beta,D)}{\mathcal{R}(D,B\backslash D)}\\  &=\sum_{D\subset
B\backslash\{\beta\}, |D|=k}\RR(A,\beta)\mathcal{R}(A,B\backslash ( \beta  \cup D))\,
\,\frac{\mathcal{R}(\beta,D)}{\RR(D,\beta)\mathcal{R}(D,B\backslash (\beta \cup D))} \\
& = (-1)^{m-k} f(\beta)\,\sum_{D\subset
B\backslash\{\beta\}, |D|=k} \frac{\mathcal{R}(A,B\backslash ( \beta  \cup D))}{\mathcal{R}(D,B\backslash (\beta \cup D))} \\ &=
(-1)^{m-k} f(\beta)\,\sum_{B^{\prime}\subset
B\backslash \{\beta\}, |B^{\prime}|=d-m}
\,\frac{\mathcal{R}(A,B')}{\mathcal{R}(B\backslash ( \beta \cup  B^{\prime}),B^{\prime})}\\
&=
(-1)^{m-k+k(d-m)} f(\beta)\,\sum_{B^{\prime}\subset
B\backslash \{\beta\}, |B^{\prime}|=d-m}
\,\frac{\mathcal{R}(A,B')}{\mathcal{R}(B^{\prime},B\backslash ( \beta \cup  B^{\prime}))},\end{align*}
setting $B':=B\backslash(\beta\cup D)$. \\Therefore $\Syl_{m,m-d}(A,B)(\beta)=  (-1)^{m-k+(k+1)(d-m)}\Syl_{0,k}(A,B)(\beta)$.
The statement follows from $m-k+(k+1)(d-m)\equiv (d-m)n+m+n-1 \pmod 2$.
\end{proof}

The previous proposition allows to recover expressions in roots for the polynomials $F_k(f,g)$ and $G_k(f,g)$ appearing in   Bezout Identity \eqref{Bezout}. Note that these identities already appeared in \cite[Art. 29]{sylv}, and more recently in \cite{KS12,DKS15}.
\begin{corollary} \label{FkGk} Let $0\le k\le \min \{m-1,n-1\}$. Then
\begin{align*}F_k(f,g)&= (-1)^{m-k}\sum_{B'\subset B, |B'|=k+1}\RR(A,B\backslash B')\frac{\RR(x,B\backslash B')}{\RR(B',B\backslash B')}
\\
G_k(f,g)&=(-1)^{m-k+1}\sum_{A'\subset A, |A'|=k+1}\RR(A\backslash A',B)\frac{\RR(x,A\backslash A')}{\RR(A\backslash A',A')}.\end{align*}
\end{corollary}

\begin{proof} The identity for $F_k(f,g)$ is simply obtained by comparing $\Syl_{m,d-m}(A,B)$ and $F_k(f,g)\,f$ and simplifying $f$ in both sides:
\begin{align*} F_k(f,g)& = (-1)^{(d-m)n+m+n-1}\sum_{B^{\prime}\subset
B, |B^{\prime}|=d-m}\mathcal{R}(A,B^{\prime})\,
\,\frac{\mathcal{R}(x,B^{\prime})}{\mathcal{R}(B^{\prime},B\backslash B^{\prime})}\\
&= (-1)^{d-n+1}\sum_{B^{\prime}\subset
B, |B^{\prime}|=d-m}\mathcal{R}(A,B^{\prime})\,
\,\frac{\mathcal{R}(x,B^{\prime})}{\mathcal{R}(B\backslash B^{\prime}, B^{\prime})}\\
&= (-1)^{m-k} \sum_{D\subset
B, |D|=k+1}\mathcal{R}(A,B\backslash D)\,
\,\frac{\mathcal{R}(x,B\backslash D)}{\mathcal{R}(D,B\backslash D)}.
\end{align*}
On the other hand, we have that $$G_k(f,g) =(-1)^{(m-k)(n-k)} F_k(g,f)$$ which implies
\begin{align*}G_k(f,g)&= (-1)^{(m-k+1)(n-k)}\sum_{A'\subset A, |A'|=k+1}\RR(B,A\backslash A')\frac{\RR(x,A\backslash A')}{\RR(A',A\backslash A')}\\
&= (-1)^{m-k+1} \sum_{A'\subset A, |A'|=k+1}\RR(A\backslash A', B)\frac{\RR(x,A\backslash A')}{\RR(A\backslash A', A')}.\end{align*}
\end{proof}

Propositions \ref{dbig}, \ref{ss=sub} and \ref{mq} then imply
\begin{corollary} Let $0\le p\le m$, $0\le q\le n$  and set $d:=p+q$. Assume   $\max\{m,n\}\le d\le m+n-1$. Then
\[
\Syl_{p,q}(A,B)=  (-1)^{(d-m)(n-q)+d-n} \Big( \binom{k}{n-q} \Sres_k(f,g) -  \binom{k+1}{m-p} F_k\,f\Big),
\]
where $k:=m+n-d-1$.
\end{corollary}

\begin{proof} We have
\begin{align*}\Syl_{p,q}(A,B)& = (-1)^c \binom{k}{n-q} \Syl_{0,k}(A,B) +(-1)^e \binom{k+1}{m-p} \Syl_{m,d-m}(A,B)\\
& =(-1)^c \binom{k}{n-q} \Sres_k(f,g) + (-1)^e  (-1)^{(d-m)n+m+n-1} \binom{k+1}{m-p}F_k(f,g)\,f ,
\end{align*}
where $c:=(d-m)(n-q)+d-n$ and $e:=(d-m)(q+1)$.
The statement follows from $e+(d-m)n+m+n\equiv (d-m)(n-q)+d-n\pmod 2$.
\end{proof}

The following proposition is another application of the Exchange Lemma \ref{exchange}, and gives simple identities for the polynomials $F_k$ and $G_k$ in the Bezout Identity (\ref{Bezout}) in terms of the roots, which enable us to make a further connection with Schur polynomials.

\begin{proposition} \label{FkGk2}Let $0\le k\leq \min\{m-1,n-1\}$.
Then
\begin{eqnarray*}
F_k&=& (-1)^{k(m-k)}  \sum_{C'\subset A\cup\{x\}, |C'|=k+1}\frac{R((A\cup\{x\})\backslash C', B)}{R(C', (A\cup\{x\})\backslash C')},\\
G_k&=& (-1)^{m(n-k)} \sum_{D'\subset B\cup\{x\}, |D'|=k+1} \frac{R((B\cup\{x\})\backslash D', A)}{R(D', (B\cup\{x\})\backslash D')}.
\end{eqnarray*}
\end{proposition}

\begin{proof} These identities are also an immediate  consequence of Exchange Corollary \ref{exchange2}, applied to $r=0$ variables.  We set $C:=A\cup\{x\}$, $D:=B\cup\{x\}$.
\begin{align*}
F_k(f,g)&=  (-1)^{m-k}\sum_{B'\subset B, |B'|=k+1} \frac{R(C, B \backslash  B')}{R(B', B\backslash  B')}\\
&=  (-1)^{(m-k)(n-k)}\sum_{B'\subset B, |B'|=k+1} \frac{R(B\backslash  B',C)}{R( B\backslash  B', B')}\\
&=  (-1)^{(m-k)(n-k)}\sum_{C'\subset C, |C'|=k+1} \frac{R(B,C\backslash C')}{R( C', C\backslash C')}\\
&= (-1)^{k(m-k)} \sum_{C'\subset C, |C'|=k+1} \frac{R(C\backslash C',B)}{R(  C',C\backslash C')},\\
G_k(f,g)&= (-1)^{(m-k)(n-k)}F_k(g,f)\\
&= (-1)^{m(n-k)} \sum_{D'\subset D, |D'|=k+1} \frac{R(D\backslash D',A)}{R(  D',D\backslash D')}.
\end{align*}

\end{proof}

As a corollary, we can express $F_k(f,g)$ and $G_k(f,g)$ as symmetric polynomials in the variables $A\cup \{x\}$ and $B\cup \{x\}$, respectively.

\begin{corollary}\label{Schur}
Let $0\le k\leq \min\{m-1,n-1\}$.  Then we have the following expressions for $F_k(f,g)$ and $G_k(f,g)$ as symmetric polynomials  in $A\cup \{x\}$ and $B\cup \{x\}$, respectively:
{\small\begin{eqnarray*}
F_k(f,g)= (-1)^{m-k} \frac{\det \begin{array}{|cccc|c}
\multicolumn{4}{c}{ \scriptstyle{m+1}}&\\
\cline{1-4}\
g(\alpha_1)\alpha_1^{m-k-1}& \dots &g(\alpha_m)\alpha_m^{m-k-1}&g(x)x^{m-k-1}&\\
\vdots & & \vdots & &\scriptstyle{m-k}\\
g(\alpha_1)& \dots &g(\alpha_m)&g(x)& \\
\cline{1-4} \
\alpha_1^{k}& \dots &\alpha_m^{k}&x^{k}&\\
\vdots & & \vdots & &\scriptstyle{k+1}\\
1& \dots &1&1& \\
\cline{1-4} \multicolumn{2}{c}{}
\end{array}}{\det(V(A\cup \{x\}))}.
\end{eqnarray*}}
and
{\small $$
G_k(f,g)= (-1)^{(m-k-1)(n-k)} \frac{\det \begin{array}{|cccc|c}
\multicolumn{4}{c}{ \scriptstyle{n+1}}&\\
\cline{1-4}\
f(\beta_1)\beta_1^{n-k-1}& \dots &f(\beta_n)\beta_n^{n-k-1}&f(x)x^{n-k-1}&\\
\vdots & & \vdots & &\scriptstyle{n-k}\\
f(\beta_1)& \dots &f(\beta_n)&f(x)& \\
\cline{1-4} \
\beta_1^{k}& \dots &\beta_n^{k}&x^{k}&\\
\vdots & & \vdots & &\scriptstyle{k+1}\\
1& \dots &1&1& \\
\cline{1-4} \multicolumn{2}{c}{}
\end{array}}{\det(V(B\cup \{x\}))},
$$}
\end{corollary}

\begin{proof}
We  prove the claim for $F_k(f,g)$, the claim for $G_k(f,g)$ follows from $G_k(f,g)= (-1)^{(m-k)(n-k)}F_k(g,f)$. \\
 First we verify that the expression in the right-hand side is a polynomial: the denominator divides the numerator since making $\alpha_i=\alpha_j$ or $\alpha_i=x$ yields the vanishing of the numerator. It is symmetric in $A\cup \{x\}$ because
both the numerator and the denominator are alternate. We set $C:=A\cup \{x\}$. We expand the determinant in the numerator of the right-hand side by the first  $m-k$ rows, and  get that the ratio of determinants equals
\begin{eqnarray*}&&\frac{1}{V(C)}\sum_{C''\subset C, |C''|=m-k} \sg(C'',C)V(C\backslash C'')V(C'')\, \prod_{c\in C''} g(C'')\\
&=& \sum_{C''\subset C, |C''|=m-k}\frac{\prod_{c\in C''} g(C'')}{R(C'', C\backslash C'')}\\
&=&  \sum_{C'\subset C, |C'|=k+1}\frac{\RR(C\backslash C', B)} {\RR(C\backslash C', C')} \ = \ (-1)^{m-k} F_k(f,g),
\end{eqnarray*}
using for the  first equality that $$\det(V(C))=\sg(C'',C)\det(V(C''))\det(V(C\backslash C''))\RR(C'', C\backslash C''),$$ where $\sg(C'',C)$ is the sign needed to bring the columns of   $C''$ into $\{1, \ldots, m-k\}$.
\end{proof}

We close the paper by pointing  to a simple connection of the previous corollary to Schur polynomials. Recall that the Schur polynomial $s_\lambda({\bfx})$ corresponding to  an $\ell$-tuple of indeterminates $\bfx=(x_1,\dots,x_\ell)$ and a partition $\lambda=(\lambda_1,\ldots, \lambda_\ell)\in \Z^\ell_{\geq 0}$, where $\lambda_1\ge \lambda_2\ge \cdots \ge \lambda_\ell\ge 0$, is defined as
$$
s_\lambda({\bfx}):=\frac{\det \begin{array}{|ccc|c}
\multicolumn{3}{c}{\scriptstyle{\ell}}&\\
\cline{1-3}\
x_1^{\lambda_1+\ell-1}& \dots &x_\ell^{\lambda_1+\ell-1}&\\
x_1^{\lambda_2+\ell-2}& \dots &x_\ell^{\lambda_2+\ell-2}&\\
\vdots & & \vdots & \scriptstyle{\ell}\\
x_1^{\lambda_{\ell-1}+1}& \dots &x_\ell^{\lambda_{\ell-1}+1}& \\
x_1^{\lambda_{\ell}}& \dots &x_\ell^{\lambda_{\ell}}&\\
\cline{1-3} \multicolumn{2}{c}{}
\end{array}}{\det(V({\bfx}))},
$$
or equivalently, it is the determinant of the submatrix of the Vandermonde matrix $V_{\lambda_1+\ell}({\bfx})$  corresponding to  the rows indexed by the exponents $\lambda_1+\ell-1, \lambda_2+\ell-2,\ldots, \lambda_{\ell}$, divided by the usual Vandermonde determinant $\det(V({\bfx}))$.\\

To see the connection between  Corollary \ref{Schur} and Schur polynomials, assume first that $g=x^n$. Then we have
$$
F_k(f,x^n)= (-1)^{m-k} \frac{\det \begin{array}{|cccc|c}
\multicolumn{4}{c}{ \scriptstyle{m+1}}&\\
\cline{1-4}\
\alpha_1^{n+m-k-1}& \dots &\alpha_m^{n+m-k-1}&x^{n+m-k-1}&\\
\vdots & & \vdots & &\scriptstyle{m-k}\\
\alpha_1^n& \dots &\alpha_m^n&x^n& \\
\cline{1-4} \
\alpha_1^{k}& \dots &\alpha_m^{k}&x^{k}&\\
\vdots & & \vdots & &\scriptstyle{k+1}\\
1& \dots &1&1& \\
\cline{1-4} \multicolumn{2}{c}{}
\end{array}}{\det(V(A\cup \{x\}))}.
$$
Thus, by definition, we immediately get that
$$F_k(f,x^n)= (-1)^{m-k} s_\lambda(A\cup\{x\}), $$
where $\lambda:=(n-k-1, \ldots, n-k-1, 0,\ldots, 0)=((n-k-1)^{m-k};0^{k+1})\in \Z_{\geq 0}^{m+1}$.\\
Similar expressions can be obtained for $G_k(x^m, g)$ in terms of the Schur polynomial on $B\cup \{x\}$ with respect to the partition $((m-k-1)^{n-k};0^{k+1})$.\\
For general $g=\sum_{i=0}^n g_ix^i$ we can use the following matrix identity, together with the Cauchy-Binet formula, to get an expression for $F_k(f,g)$  in terms of Schur polynomials on $A\cup \{x\}$:
{\small \begin{eqnarray*}
&&\begin{array}{|cccc|c}
\multicolumn{4}{c}{ \scriptstyle{m+1}}&\\
\cline{1-4}\
g(\alpha_1)\alpha_1^{m-k-1}& \dots &g(\alpha_m)\alpha_m^{m-k-1}&g(x)x^{m-k-1}&\\
\vdots & & \vdots & &\scriptstyle{m-k}\\
g(\alpha_1)& \dots &g(\alpha_m)&g(x)& \\
\cline{1-4} \
\alpha_1^{k}& \dots &\alpha_m^{k}&x^{k}&\\
\vdots & & \vdots & &\scriptstyle{k+1}\\
1& \dots &1&1& \\
\cline{1-4} \multicolumn{2}{c}{}
\end{array}=\\
&=&\begin{array}{c|ccccc|}
\multicolumn{6}{c}{\scriptstyle{m+n-k}}\\
\cline{2-6}
&g_{n} & \dots & g_{0} &  & \\
\scriptstyle{m-k}& & \ddots &  & \ddots & \\
& &  & g_{n} & \dots & g_{0} \\
\cline{2-6}
&&&&&\\
\scriptstyle{k+1}& &\multicolumn{2}{c}{{\bf 0}_{(k+1)\times (m+n-1)}}& &{\bf Id}_{k+1}\\
&&&&&\\
\cline{2-6} \multicolumn{2}{c}{}
\end{array}\cdot V_{m+n-k}(A\cup \{x\}) . \end{eqnarray*}}

Note that the above expressions using Schur polynomials on $A\cup \{x\}$ are related to the expressions given in \cite[page 3]{Las}, where the $k=m-1$ case is considered. In the $k=m-1$ case,  the Lagrange operator defined in \cite[page 3]{Las} is the map $L_A:g\mapsto F_{m-1}(f,g)$, using our notation. For general $k$, a special case of the Sylvester operator defined in  \cite[page 16]{Las} is the map $g(x_1)\cdots g(x_{m-k})\mapsto F_k(f,g)$.

\bigskip

\begin{ac} Agnes Szanto and Teresa Krick thank the Simons Institute for the Theory of Computing, for the  Fall'14 program  ``Algorithms and Complexity in Algebraic Geometry'' where this work was started.
We are also grateful to  Carlos D'Andrea for the  many useful discussions we had with him,  to Giorgio Ottaviani for a great conversation on symmetric polynomials  which improved the proof of Lemma \ref{dim}, and to   Ricky Ini Liu  who helped us understanding the connections between Schur polynomials and Proposition \ref{FkGk2}.
\end{ac}

\bigskip
\bibliographystyle{alpha}
\def\cprime{$'$} \def\cprime{$'$} \def\cprime{$'$}

\end{document}